\documentclass[amscd,amssymb,11pt]{amsart}

\usepackage{amssymb}
\usepackage[all]{xy}

\newtheorem{thm}{Theorem}[section]
\newtheorem{lem}[thm]{Lemma}
\newtheorem{cor}[thm]{Corollary}
\newtheorem{prop}[thm]{Proposition}

\newtheorem{prop/def}[thm]{Proposition/Definition}

\theoremstyle{definition}

\newtheorem{defn}[thm]{Definition}

\newtheorem{ex}[thm]{Example}

\theoremstyle{remark}

\newtheorem{rem}[thm]{Remark}

\numberwithin{equation}{section}

\newcommand{\thmref}[1]{Theorem~\ref{#1}}
\newcommand{\corref}[1]{Corollary~\ref{#1}}
\newcommand{\secref}[1]{\S\ref{#1}}

\newcommand{\propref}[1]{Proposition~\ref{#1}}

\newcommand{\Hom}{\operatorname{Hom}}

\newcommand{\Ext}{\operatorname{Ext}}
\newcommand{\Rep}{\operatorname{Rep}}
\newcommand{\Sing}{\operatorname{Sing}}
\newcommand{\Inj}{\operatorname{Inj}}
\newcommand{\Surj}{\operatorname{Surj}}
\newcommand{\Epi}{\operatorname{Epi}}
\newcommand{\Fin}{\operatorname{Fin}}
\newcommand{\Iso}{\operatorname{Iso}}
\newcommand{\End}{\operatorname{End}}

\newcommand{\Char}{\operatorname{char}}

\newcommand{\A}{{\mathcal  A}}

\newcommand{\C}{{\mathcal  C}}
\newcommand{\M}{{\mathcal  M}}

\newcommand{\K}{{\mathcal  K}}

\newcommand{\V}{{\mathcal V}}

\newcommand{\F}{{\mathbb  F}}

\newcommand{\Z}{{\mathbb  Z}}

\newcommand{\PP}{{\mathcal P}}
\newcommand{\QQ}{{\mathcal Q}}

\newcommand{\ra}{\rightarrow}
\newcommand{\xra}{\xrightarrow}

\begin{document}

\title[Generic representations in nondescribing characteristic]{Generic representation theory of finite fields in nondescribing characteristic}

\author[Kuhn]{Nicholas J.~Kuhn}
\address{Department of Mathematics \\ University of Virginia \\ Charlottesville, VA 22904}
\email{njk4x@virginia.edu}
\thanks{This research was partially supported by grants from the National Science Foundation}

\date{May 1, 2014.}


\begin{abstract}  Let $\Rep(\F;K)$ denote the category of functors from finite dimensional $\F$--vector spaces to $K$--modules, where $\F$ is a field and $K$ is a commutative ring.  We prove that, if $\F$ is a finite field, and $\Char \F$ is invertible in $K$, then the $K$--linear abelian category $\Rep(\F;K)$ is equivalent to the product, over all $n\geq 0$, of the categories of $K[GL_n(\F)]$--modules.

As a consequence, if $K$ is also a field, then small projectives are also injective in $\Rep(\F;K)$, and will have finite length.  Even more is true if $\Char K = 0$: the category $\Rep(\F;K)$ will be semisimple.

In a last section, we briefly discuss `$q=1$' analogues and consider representations of various categories of finite sets.

The main result follows from a 1992 result by L.G.Kovacs about the semigroup ring $K[M_n(\F)]$.
\end{abstract}

\maketitle
\section{Introduction} \label{introduction}

Let $\V(\F)$ be the category of finite dimensional vector spaces over a finite field $\F$ of characteristic $p$, and let $K$ be a commutative ring, likely a field.

Then let  $\Rep(\F;K)$ denote the category whose objects are functors
$$ F: \V(\F) \ra K\text{--modules},$$
and whose morphisms are the natural transformations.

This is a $K$--linear abelian category in the usual way.  For example,
$$ 0 \ra F \ra G \ra H \ra 0$$
is short exact in $\Rep(\F;K)$ means that, for any $V \in \V(\F)$, the sequence
$$ 0 \ra F(V) \ra G(V) \ra H(V) \ra 0$$
is a short exact sequence of $K$--modules.

Our papers \cite{genrep1} -- \cite{genrepfilt} study the case when $K=\F$.  Following terminology used in those papers,  we refer $F \in \Rep(\F;K)$ as a {\em generic representation} of the field $\F$.
To explain this, note that there are evident connections with the representation theory of the general linear groups $GL_n(\F)$, as $F \in \Rep(\F;K)$ defines a family $\{F(\F^n) \ | \ n=0,1,2, \dots\}$ of $K[GL_n(\F)]$--modules via evaluation.   There is even more structure here, as $F(\F^n)$ is a module for the semigroup ring $K[M_n(\F)]$, where $M_n(\F)$ is the semigroup of all $n\times n$ matrices over $\F$. Indeed, as described in \cite{genrep2}, a generic representation $F$ is roughly the same thing as a compatible sequence of $K[M_n(\F)]$--modules for all $n$.

There has been much success studying the case when $K=\F$:
\begin{itemize}
\item Many extension groups $\Ext^*_{\Rep(\F;\F)}(F,G)$ have been calculated when $F$ and $G$ are classical functors. See, e.g.,  \cite{fls, ffss, touze}.
\item There is a deep connection between $\Rep(\F_p;\F_p)$ and the category of unstable modules over the mod $p$ Steenrod algebra of algebraic topology.  See \cite{hls1, schwartz, genrep1, genrep3}.
\item There are connections to both algebraic $K$--theory and the representation theory of algebraic groups. See, e.g., \cite{fs, betley}.
\end{itemize}

By contrast, there has been relatively little said about the structure of $\Rep(\F;K)$ when $K$ is a field of characteristic different than $p$.   Our main theorem here remedies this.

\begin{thm} \label{main thm}  Let $\F$ be a finite field of characteristic $p$.  If $p$ is invertible in a commutative ring $K$, there is a natural equivalence of $K$--linear abelian categories
$$ \Rep(\F;K) \simeq \prod_{n=0}^{\infty} K[GL_n(\F)]\text{--modules}.$$
\end{thm}

Some structural results about $\Rep(\F;K)$ are immediate corollaries.

\begin{cor} \label{cor 1}  If $K$ is a field of characteristic different than $p$, then all  projectives in $\Rep(\F;K)$ are also injective, and indecomposable projectives have only finitely many composition factors.
\end{cor}

\begin{cor} \label{cor 2}  If $K$ is a field of characteristic 0, then $\Rep(\F;K)$ is semisimple.
\end{cor}

We remark that in $\Rep(\F;\F)$, all non-constant indecomposable projective functors have an infinite number of composition factors. These corollaries show how different the situation is in nondescribing characteristic.

Here is an example of the the computational implications of our theorem.

\begin{ex}  Let $Gr(V)$ be the set of all subspaces of a finite dimensional $\F$--vector space $V$, and let $K[Gr] \in \Rep(\F;K)$ be the functor sending $V$ to $K[Gr(V)]$, the $K$--module spanned by this set.  Suppose $p$ is invertible in $K$.  Under the correspondence of the main theorem, $K[Gr]$ corresponds to the sequence of trivial modules for the groups $GL_n(\F)$.  We conclude that there is an isomorphism of graded $K$--algebras
$$ \Ext_{\Rep(\F;K)}^*(K[Gr],K[Gr]) \simeq \prod_{n=0}^{\infty} H^*(GL_n(\F);K).$$
It is interesting to note that much about the right side of this isomorphism was computed by Quillen \cite{quillen}.
\end{ex}

The theorem will be proved in \secref{main thm proof sec}.  In some sense, the proof is formal, starting from a result of L.G.Kov\'acs \cite{kovacs} about the semigroup ring $K[M_n(\F)]$.

In \secref{related results section}, we discuss Kov\'acs' theorem and some related formulations of both his work and our main theorem.

In \secref{finite set sec}, we briefly discuss `$q=1$' analogues and consider representations of various categories of finite sets.  \\

\noindent{\bf Acknowledgements} Steven Sam has recently verified a long conjectured structural property of $\Rep(\F;\F)$ -- it is generated by Noetherian projectives -- and his proof works without change for $\Rep(\F;K)$ for arbitrary Noetherian rings $K$.  It was talks with Sam about generic representation theory that inspired the project here.

This research was done during a stay at the Mathematical Sciences Research Institute in Berkeley during the spring 2014 algebra topology program.  The author is grateful for the opportunity to be part of this program which was supported by N.S.F. grant 0932078000, and also for support through N.S.F. grant 0967649.

\section{Proof of the main theorem} \label{main thm proof sec}

\subsection{A theorem of L.G.Kov\'acs}   The key input of our proof is an elegant 1992 result of L.G.Kov\'acs \cite{kovacs}.  This followed a related announcement by Faddeev \cite{faddeev}, and related results by Okni\'nski and Putcha \cite{okninski putcha}.  Kov\'acs also thanks W.~D.~Munn for conversations about this topic, and mentions that Munn had a proof (apparently unpublished) of precisely the result we state below.

To state this, let $\Sing_n(\F) \subset M_n(\F)$ denote the set of singular matrices.   The $K$--linear span of this set $K[\Sing_n(\F)]$ is a two sided ideal in the semigroup ring $K[M_n(\F)]$, and the quotient identifies with $K[GL_n(\F)]$. Thus one has a short exact sequence
\begin{equation} \label{M G sequence}
0 \ra K[\Sing_n(\F)] \ra K[M_n(\F)] \ra K[GL_n(\F)] \ra 0.
\end{equation}

\begin{thm} \cite{kovacs} If $p$ is invertible in $K$, $K[\Sing_n(\F)]$ contains an idempotent $e_n^S$ that serves as a unit.  Thus
the short exact sequence (\ref{M G sequence}) splits as a sequence of unital $K$--algebras.
\end{thm}

Note that $e_n^S$ as in this theorem is necessarily unique.  One easily sees that such an $e_n^S$ satisfies the following properties (see \secref{kovacs example}).
\begin{itemize}
\item $e_n^S$ is central in $K[M_n(\F)]$.
\item $e_n^S$ is  fixed under the transpose automorphism of $K[M_n(\F)]$.
\item  $e_n^S$ is  fixed under conjugation by any element of $GL_n(\F)$.
\end{itemize}
In \secref{kovacs example}, we will explicitly describe $e_2^S$, when $\F = \F_2$.

For the rest of the section, we assume that $p = \Char \F$ is invertible in $K$.

We name the complementary idempotent:

\begin{defn}  Let $e_n^G = 1-e_n^S \in K[M_n(\F)]$.
\end{defn}

$e_n^G$ is a central idempotent satisfying the following properties.
\begin{itemize}
\item $e_n^GK[M_n(\F)]e_n^G \simeq K[GL_n(\F)]$ as algebras.
\item $e_n^G\cdot [A] = 0$ for all $A \in \Sing_n(\F)$.
\end{itemize}

\subsection{An old family of projective generators}

\begin{defn}  Let $P_n \in \Rep(\F;K)$ be defined by letting $P_n(V) = K[\Hom(\F^n,V)]$.
\end{defn}

Yoneda's lemma tells us:
\begin{itemize}
\item There is a natural isomorphism
$$ \Hom_{\Rep(\F;\K)}(P_n,F) \simeq F(\F^n)$$
for all $\F \in \Rep(\F;\K)$.
\end{itemize}
Let $\PP = \{P_n \ | \ n=0,1,2,\dots\}$.  As every object in $\V(\F)$ is isomorphic to $\F^n$ for some $n$, we deduce:
\begin{itemize}
\item  $\PP$ is family of projective generators for $\Rep(\F;K)$.
\end{itemize}

\subsection{A new family of projective generators}

The algebra $\End_{\Rep(\F;K)}(P_n)$ identifies with $K[M_n(\F)]$, and we make the following definition.

\begin{defn}  Let $P_n^G = P_n \cdot e_n^G$, a direct summand of the projective $P_n$.
\end{defn}

Let $gr_k(n)$ be the number of $k$--dimensional subspaces of $\F^n$.  (If $\F=\F_q$, this is commonly denoted $\begin{bmatrix} n\\k \end{bmatrix}_q$.)

\begin{prop} \label{prop 1}  There is an isomorphism of functors
$$ \bigoplus_{k=0}^n gr_k(n) P_k^G \simeq P_n.$$
Thus $\PP^G = \{P_n^G \ | \ n=0,1,2, \dots\}$ is a set of projective generators for $\Rep(\F;K)$.
\end{prop}

We defer the proof to later in the section.

\begin{prop} \label{prop 2}
\begin{equation*}
\Hom_{\Rep(\F;K)}(P^G_m, P^G_n) =
\begin{cases}
K[GL_n(\F)] & \text{if } m=n \\ 0 & \text{if } m \neq n.
\end{cases}
\end{equation*}
\end{prop}

We also defer the proof of this.

Since $\End_{\Rep(\F;K)}(P_n) = K[M_n(\F)]$, the two propositions combine to prove what was stated as the main theorem in \cite{kovacs}.

\begin{cor} \label{KM decomp cor}  If $p$ is invertible in $K$, there is an isomorphism of $K$--algebras
$$ K[M_n(\F)] \simeq \prod_{k=0}^n M_{gr_k(n)}(K[GL_k(\F)]).$$
\end{cor}

\subsection{Proof of \thmref{main thm}}

The main theorem will follow immediately from the last two propositions, using the `multiobject' version of classic Morita equivalence.

One has the following general situation.  Suppose $\QQ$ is a set of objects in a $K$--linear abelian category $\A$.  Let $\End(\QQ)$ be the full subcategory of $\A$ whose set of objects is $\QQ$.  Then let $\End(\QQ)\text{--mod}$ denote the category of contravariant $K$--linear functors from $\End(\QQ)$ to $K\text{--mod}$.

A functor
$$\Theta: \A \ra \End(\QQ)\text{--mod}$$
is defined as follows.  Given $A \in \A$ and $Q \in \QQ$, let $$ \Theta(A)(Q) = \Hom_{\A}(Q,A).$$

The functor $\Theta$ is always the right half of an adjoint pair $(\Psi,\Theta)$ where $\Psi(\M) = \M \otimes_{\End(\QQ)} \QQ$, suitably interpreted.

The multiobject Morita theorem goes as follows.

\begin{lem}  If $\QQ$ is set of small projective generators for $\A$, then $\Psi$ and $\Theta$ are inverse equivalences of $K$--linear abelian categories.
\end{lem}

\thmref{main thm} now follows by letting $\A = \Rep(\F;K)$ and $\QQ = \PP^G$.  \propref{prop 1} tells us that $\PP^G$ is a set of small projective generators for $\Rep(\F;K)$, so that
$$\Rep(\F;K) \simeq \End(\PP^G)\text{--mod},$$
and then \propref{prop 2} shows that
$$\End(\PP^G)\text{--mod} \simeq \prod_{n=0}^{\infty} K[GL_n(\F)]\text{--mod}.$$

It remains to prove the two propositions.

\subsection{Proof of \propref{prop 1}}
Let $\Inj(W,V)$ and $\Surj(W,V)$ be the sets of injective and surjective linear maps in $\Hom(W,V)$.
Then $K[\Inj(\F^k,V)]$ is a free right $K[GL_k(\F)]$--module and $K[\Surj(W,\F^k)]$ is a free left $K[GL_k(\F)]$--module.

\begin{lem}  Composition of linear maps induces an isomorphism of $K$--modules
$$ \bigoplus_{k=0}^n K[\Inj(\F^k,V)] \otimes_{K[GL_k(\F)]} K[\Surj(\F^n,\F^k)] \xra{\sim} K[\Hom(\F^n,V)].$$
\end{lem}

Assuming $p$ is invertible in $K$, we can use our idempotents $e_k^G \in K[M_k(\F)]$ to write this isomorphism in a way that exhibits naturality in $V$.

The inclusion
$\Inj(\F^k,V) \subset \Hom(\F^k,V)$
induces an isomorphism of right $K[GL_k(V)]$--modules
$$ K[\Inj(\F^k,V)] \xra{\sim} K[\Hom(\F^k,V)]e_k^G$$

The isomorphism of the lemma rewrites as a natural isomorphism
$$ \bigoplus_{k=0}^n K[\Hom(\F^k,V)]e_k^G \otimes_{K[GL_k(\F)]} K[\Surj(\F^n,\F^k)] \xra{\sim} K[\Hom(\F^n,V)].$$

Otherwise said, we have verified the next corollary.
\begin{cor}  There is an isomorphism of functors
$$ \bigoplus_{k=0}^n P_k^G \otimes_{K[GL_k(\F)]} K[\Surj(\F^n,\F^k)] \xra{\sim} P_n.$$
\end{cor}

Now we decompose $K[\Surj(\F^n,\F^k)]$.  The kernel (a.k.a. nullspace) of any surjection $A: \F^n \ra \F^k$ will be a subspace of $\F^n$ of codimension $k$, and this kernel is invariant under left multiplication by elements in $GL_k(\F)$.  Thus we have the following.

\begin{lem}  There is an isomorphism of left $K[GL_k(\F)]$--modules
$$K[\Surj(\F^n,\F^k)] \simeq \bigoplus_{W} K[GL_k(\F)],$$
where the sum runs over $W < \F^n$ of codimension $k$.
\end{lem}

As $\{W < \F^n \ | \ \text{codim } W =k\}$ has cardinality $gr_k(n)$, the last lemma and corollary imply \propref{prop 1}.

\subsection{Proof of \propref{prop 2}}

Recall that each $e_n^G$ is a central idempotent in $K[M_n(\F)]$ satisfying the following two properties:
\begin{itemize}
\item $e_n^GK[M_n(\F)]e_n^G \simeq K[GL_n(\F)]$ as algebras.
\item $[B] \cdot e_n^G = 0$ for all $B \in \Sing_n(\F)$.
\end{itemize}

By construction, we have isomorphisms
$$ \Hom_{\Rep(\F;K)}(P_m^G,P_n^G) = e_m^GK[M_{m,n}(\F)]e_n^G,$$
where we have identified $\Hom(\F^n,\F^m)$ with $M_{m,n}(\F)$, the set of $m\times n$ matrices over $\F$.

Thus our first property says that $\End_{\Rep(\F;K)}(P_m^G)\simeq K[GL_n(\F)]$.

If $m<n$, then every element in $e_m^GK[M_{m,n}(\F)]$ can be written as a linear combination of terms of the form $[AB]$, where $A \in M_{m,n}(\F)$ and $B \in \Sing_n(\F)$.  The second property above tells us that for such $A$ and $B$, $[AB]\cdot e_n^G = [A]([B]\cdot e_n^G) = 0$, so that $\Hom_{\Rep(\F;K)}(P_m^G,P_n^G) = 0$.

There is a similar proof when $m>n$; this case also follows from the one just proved by using transpose of matrices.

\section{Further remarks and related results} \label{related results section}

\subsection{On Kov\'acs' theorem} \label{kovacs example}

We make a few observations as an aid to those readers who might wish to better understand Kov\'acs' proof that $K[\Sing_n(\F)]$ contains a unit.  These are also explicitly or implicitly said in \cite{kovacs}.

Let $e_{n-1} = [I_{n-1}] \in K[\Sing_n(\F)]$ where $I_{n-1}$ is the $n \times n$ matrix which has 1's on the first $(n-1)$ diagonal entries and is zero elsewhere.

\begin{lem}  An element $e \in K[\Sing_n(\F)]$ is two sided unit if and only if it satisfies the following two properties:
\begin{itemize}
\item $e$ is invariant under the conjugation action of $GL_n(\F)$ on $K[\Sing_n(\F)]$.
\item $ee_{n-1} = e_{n-1}$.
\end{itemize}
\end{lem}
\begin{proof}  Suppose that $e$ satisfies the two properties.  Any $A \in \Sing_n(\F)$ admits a decomposition of the form $A=BC$, where $B$ is an idempotent of rank $n-1$ and so is conjugate to $I_{n-1}$.  The two properties imply that $e[B] = [B]$, and thus $e[A] = (e[B])[C] = [B][C] = [A]$.  In other words, $e$ is a left unit for $K[\Sing_n(\F)]$.

Taking the transpose is an antiautomorphism of $K[\Sing_n(\F)]$, so $e^T$ is a right unit for $K[\Sing_n(\F)]$.  But then $e=e^T$ is a two sided unit.

The other implication of the lemma is immediate.
\end{proof}

Elements in $K[\Sing_n(\F)]^{GL_n(\F)}$ are linear combinations of orbit sums.  Kov\'acs discovers that one only needs to use conjugacy classes of matrices he terms `semi-idempotent': matrices $E$ that act as the identity on the image of $E^N$ for $N>>0$.  Combined with the last lemma, one is well on one's way to finding the unit $e \in K[\Sing_n(\F)]$.

\begin{ex}  We show how this all works when $n=2$ and $\F=\F_2$, the field with two elements.  There are three relevant orbit sums, and we find ourselves looking for $a,b,c \in \Z[\frac{1}{2}]$ such that
\begin{equation*}
\begin{split}
e & = a\left\{\begin{bmatrix} 1&0 \\ 0&0 \end{bmatrix} +   \begin{bmatrix} 1&1 \\ 0&0 \end{bmatrix} +   \begin{bmatrix} 0&0 \\ 1&1 \end{bmatrix} +   \begin{bmatrix} 1&0 \\ 1&0 \end{bmatrix} +   \begin{bmatrix} 0&0 \\ 0&1 \end{bmatrix} +   \begin{bmatrix} 0&1 \\ 0&1 \end{bmatrix} \right\} \\
  & +  b\left\{\begin{bmatrix} 0&0 \\ 1&0 \end{bmatrix} + \begin{bmatrix} 1&1 \\ 1&1 \end{bmatrix} + \begin{bmatrix} 0&1 \\ 0&0 \end{bmatrix} \right\} + c \begin{bmatrix} 0&0 \\ 0&0 \end{bmatrix}
\end{split}
\end{equation*}
satisfies $e\begin{bmatrix} 1&0 \\ 0&0 \end{bmatrix} = \begin{bmatrix} 1&0 \\ 0&0 \end{bmatrix}$.
\end{ex}

This leads to the system of equations
\begin{equation*}
\begin{split}
1 & = 2a \\
0 & = a+b \\
0 & = 2a+b+c,
\end{split}
\end{equation*}
which has solution $a=1/2$, $b=-1/2$, and $c=-1/2$.

\subsection{Split recollement}  As described in \cite{genrep2}, for any fields $K$ and $\F$, one has a recollement diagram:
$$ K[GL_n(\F)]\text{--mod}
\begin{array}{c} \stackrel{q}{\longleftarrow} \\[-.08in]
\stackrel{i}{\longrightarrow} \\[-.08in] \stackrel{p}{\longleftarrow}
\end{array}
K[M_n(\F)]\text{--mod}
\begin{array}{c} \stackrel{l}{\longleftarrow} \\[-.06in] \stackrel{e}{\longrightarrow} \\[-.08in]
\stackrel{r}{\longleftarrow}
\end{array} K[M_{n-1}(\F)]\text{--mod}.$$
In this diagram, the functors $i$ and $e$ are exact, with left adjoints $q$ and $l$, and right adjoints $p$ and $r$.  The natural maps
$$ l(e(N)) \ra N \text{ and } N \ra e(r(N))$$
are both isomorphisms for all $K[M_{n-1}(\F)]$--modules $N$.

Explicitly, $i$ is the pullback of modules via the quotient map $K[M_n(\F)] \ra K[GL_n(\F)]$, while
$$e(M) = e_{n-1}M,$$
$$l(N) = K[M_n(\F)]e_{n-1}\otimes_{K[M_{n-1}(\F)]}N,$$
 and
$$r(N) = \Hom_{K[M_{n-1}(\F)]}(e_{n-1}K[M_n(\F)],N).$$

Kov\'acs' theorem is easily seen to be equivalent to the statement that, when $\F$ is a finite field of of characteristic $p$ and $\frac{1}{p} \in K$, $l(N) \simeq r(N)$ and thus are both exact, and similarly $q(M) \simeq p(M)$.  It follows that the assigment $M \mapsto (e(M),q(M))$ induces an equivalence of $K$--linear abelian categories
$$ K[M_n(\F)]\text{--mod} \simeq K[M_{n-1}(\F)]\text{--mod} \times K[GL_n(\F)]\text{--mod}.$$

\subsection{Splitting the rank filtration}  Similarly, if $e_n(F) = F(\F^n)$,
$$e_n: \Rep(\F;K) \ra K[M_n(\F)]\text{--mod}$$
 is an exact functor with left and right adjoints $l_n$ and $r_n$ defined by letting
$$l_n(M) = P_n \otimes_{K[M_n(\F)]} M$$
and
$$ r_n(M)(V) = \Hom_{K[M_n(\F)]}(K[\Hom(V,\F^n)],M).$$

The counit of the $(l_n,e_n)$ adjunction takes the form
$$ P_n \otimes_{K[M_n(\F)]} F(\F^r) \ra F,$$
and we let $F^n \subseteq F$ be the image of this map.

This defines a canonical increasing rank filtration of a generic representation $F$:
$$ F^0 \subseteq F^1 \subseteq F^2 \subseteq \dots \subseteq \bigcup_{n=0}^{\infty} F^n = F.$$

A reinterpretation of our main theorem goes as follows.  If $\frac{1}{p} \in K$, then $l_n(M) \simeq r_n(M)$, both $l_n$ and $r_n$ are exact, and the rank filtration splits: there is a natural decomposition
$$ F \simeq \bigoplus_{k=0}^{\infty} F^k/F^{k-1}.$$

\section{Generic representations of finite sets, and a $q=1$ analogue.}  \label{finite set sec} One might wonder to what extent the analogues of our results here are true if one considers representations of finite sets, rather than finite dimensional vector spaces over $\F_q$.

There are various categories of finite sets one might consider, two of which are the categories of based finite sets, $\Gamma$, and unbased finite sets, $\Fin$.  The analogues of our results don't hold for either of these, as the next examples show.

If $\C$ is a small category, we let $\Rep(\C;K)$ denote the category of covariant functors from $\C$ to $K$--mod.

\begin{ex} It is an observation of Pirashvili \cite{pirashvili} that $\Rep(\Gamma;K)$ is equivalent to $\Rep(\Epi;K)$, where $\Epi$ is the category of finite sets and epimorphisms.  We let $P_n^{\Epi}(S) = K[\Epi(\mathbf n, S)]$, where $\mathbf n = \{1,\dots,n\}$.  Then there is an inclusion $P_1^{\Epi} \ra P_2^{\Epi}$ which is never split.  Thus $\Rep(\Gamma;K)$ is not semisimple, even when $K$ is a field of characteristic 0.

A similar thing happens if one considers the category of contravariant functors $\Rep(\Gamma^{op};K)$.
\end{ex}

\begin{ex}  (We thank Steven Sam for showing us this example.) Let $P_1 \in \Rep(\Fin;K)$ be defined by $P_1^{\Fin}(S) = K[S]$ for all finite sets $S$, and, by abuse of notation, let $K$ denote the functor with constant value $K$.  There is a natural transformation $\epsilon: P_1^{\Fin} \ra K$ defined by letting $\epsilon([s])=1$ for all $s \in S$.  This is an epimorphism which is never split, even when $K$ is a field of characteristic 0.

Dually, there is a non-split natural inclusion $K \ra K^S$ of contravariant functors of $\Fin$.

\end{ex}

Now let $\Inj_*$ denote the subcategory of based sets $\Gamma$ having the same objects --  finite based sets $(S,s_0)$ -- and with morphisms equal to based maps $f: (S,s_0) \ra (T,t_0)$ that are one-to-one on the complement of $f^{-1}(t_0)$.

If we let $\Inj$ denote the category of finite sets and injections, and $\Iso$ denote the category of finite sets and bijections, then one notes that there is a decomposition
$$ \Inj_* = \Inj\circ_{\Iso} \Inj^{op}$$
of the sort discussed in \cite{slominska} or \cite{helm thesis, helmstutler}.  The main theorem of any of these applies to show the following theorem.

\begin{thm} For any commutative ring $K$, there are natural equivalences of $K$--linear abelian categories
$$ \Rep(\Inj_*;K) \simeq \Rep(\Iso;K) \simeq \prod_{n=0}^{\infty} K[\Sigma_n]\text{--mod}.$$
\end{thm}

Explicitly, see \cite[Theorem 2.5]{slominska}, or apply \cite[Theorem 7.1]{helmstutler} to either \cite[Example 3.15]{helmstutler} or \cite[Example 3.19]{helmstutler}.

Furthermore, the construction in \cite{helmstutler} yields the following analogue of \corref{KM decomp cor}.  Let $R_n = \End_{\Inj_*}(\mathbf n_*)$, where $\mathbf n_* = \{0,1,\dots,n\}$ with basepoint $0$.   Then $R_n$ is the $n$th symmetric inverse semigroup, and is also called the rook monoid in \cite{solomon}.

\begin{cor} \label{rook cor} For all commutative rings $K$, there is an isomorphism of $K$--algebras
$$ K[R_n] \simeq \prod_{k=0}^n M_{\binom{n}{k}}(K[\Sigma_k]).$$
\end{cor}

(Further discussion about this theorem, corollary, and generalizations may appear in a short note.)

\begin{rem} We note that our category $\Inj_*$ is the same as the category called `FI\#' in \cite{cef}, and \corref{rook cor} is a strengthening to all rings $K$ of the main result in \cite{solomon}.
\end{rem}

\end{document}